\documentclass[11pt]{amsart}
\usepackage{amsmath,amssymb,amsfonts,graphics}

\newtheorem{thm}{Theorem}[section]
\newtheorem{cor}[thm]{Corollary}
\newtheorem{lem}[thm]{Lemma}

\theoremstyle{definition}

\theoremstyle{remark}
\newtheorem{exm}[thm]{Example}
\newtheorem{rem}[thm]{Remark}
\numberwithin{equation}{section}

\newcommand{\R}{\mathbb{R}}
\newcommand{\N}{\mathbb{N}}
\newcommand{\Z}{\mathbb{Z}}
\newcommand{\T}{\mathbb{T}}
\newcommand{\C}{\mathbb{C}}

\newcommand{\tw}{*}

\newcommand{\Sm}{\mathcal{S}}

\newcommand{\tPh}{\tilde{\Phi}}
\newcommand{\tPs}{\tilde{\Psi}}

\newcommand{\txi}{\tilde{\xi}}

\newcommand{\om}{\omega}
\newcommand{\Om}{\Omega}
\newcommand{\sg} {\sigma}

\newcommand{\la}{\langle}
\newcommand{\ra}{\rangle}

\begin{document}

\title[Weak amenability of weighted Orlicz alg.]
{Weak amenability of weighted Orlicz algebras}

\author[Serap \"{O}ztop]{Serap \"{O}ztop}
\address{Department of Mathematics, Faculty of Science, Istanbul University, Istanbul, Turkey}
\email{oztops@istanbul.edu.tr}

\author{Ebrahim Samei}
\address{Department of Mathematics and Statistics, University of Saskatchewan, Saskatoon, Saskatchewan, S7N 5E6, Canada}
\email{samei@math.usask.ca}

\author{Varvara Shepelska}
\address{Department of Mathematics and Statistics, University of Saskatchewan, Saskatoon, Saskatchewan, S7N 5E6, Canada}
\email{shepelska@gmail.com}

\footnote{{\it Date}: \today.



The second named author was partially supported by NSERC Grant no. 409364-2015 and 2221-Fellowship Program For Visiting Scientists And Scientists On Sabbatical Leave from Tubitak, Turkey. The third named author was also partially supported by a PIMS Postdoctoral Fellowship at the University of Saskatchewan.}







\maketitle

\begin{abstract}
Let G be a locally compact abelian group, $\om:G\to (0,\infty)$ be a weight, and ($\Phi$,$\Psi$) be a complementary pair of strictly increasing continuous Young functions. We show that for the weighted Orlicz algebra $L^\Phi_\om(G)$, the weak amenability is obtained under conditions similar to the one considered in \cite{Z} for weighted group algebras. Our methods can be applied to various families of weighted Orlicz algebras, including weighted $L^p$-spaces.
\end{abstract}

\section{Introduction}

Orlicz spaces represent an important class of Banach function spaces considered in mathematical analysis. This class naturally arises as a generalization of $L^p$-spaces and contains, for example, the well-known Zygmund space $L \log^+ L$ which is a Banach space related to Hardy-Littlewood maximal functions. Orlicz spaces can also contain certain Sobolev spaces as subspaces. Linear properties of Orlicz spaces have been studied thoroughly (see \cite{rao} for example). However, until recently, little attention has been paid to their possible algebraic properties, particularly, if they are considered over translation-invariant measurable spaces. One reason might be that, on its own, an Orlicz space is rarely an algebra with respect to a natural product! For instance, it is well-known that for a locally compact group $G$, $L^p(G)$ ($1<p<\infty$) is an algebra under the convolution product exactly when $G$ is compact \cite{S}. Similar results have also been obtained for other classes of Orlicz algebras (see \cite{AM}, \cite{HKM}, \cite{S} for details).

The preceding results indicate that, in most cases, Orlicz spaces over locally compact groups are simply ``too big" to become algebras under convolution. However, it turned out that it is possible for ``weighted" Orlicz spaces to become algebras. In fact, weighted $L^p$-algebras and their properties have been studied by many authors including J. Wermer on the real line and
Yu. N. Kuznetsova on general locally compact groups (see, for example, \cite{K1}, \cite{K2}, \cite{KM}, \cite{jW} and the references therein). These spaces have various properties and numerous applications in harmonic analysis. For instance, by applying the Fourier transform, we know that Sobolev spaces $W^{k,2}(\T)$ are nothing but certain weighted $l_\om^2(\Z)$ spaces.

Recently, in \cite{OO}, A. Osan\c{c}l{\i}ol and S. \"{O}ztop considered weighted Orlicz algebras over locally compact groups and studied their properties, extending, in part, the results of \cite{K1} and \cite{K2}. In \cite{OS1} and \cite{OS2}, the first two-named authors initiated a more general approach by considering the twisted convolution coming from a 2-cocycle $\Om$ with values in $\C^*$, the multiplicative group of complex numbers. Sufficient conditions on $\Om$ were found ensuring that the twisted convolution coming from $\Om$ turns the Orlicz space to a Banach algebra or a Banach $*$-algebra
\cite[Theorems 3.3 and 4.5]{OS1}. These methods produce abundant families of Arens regular, symmetric dual Banach $*$-algebras in the form of weighted Orlicz algebras, mostly on compactly generated groups with polynomial growth (see \cite[Theorems 5.2 and 5.8]{OS1} and \cite[Theorem 4.2 and 5.3]{OS2}). Certain cohomological properties of weighted Orlicz algebras were also investigated in \cite{OS2} where it was shown that either property can not happen unless $G$ is finite \cite[Theorems 6.2 and 6.3]{OS2}.

In the present paper, we continue our investigation of the cohomology of weighted Orlicz algebras focusing on the concept of weak amenability. In~\cite{Z}, Y. Zhang provided a necessary and sufficient condition on the weight $\om$ ensuring that the weighted group algebra $(L^1_\om(G),*)$ on a locally compact abelian group $G$ is weakly amenable (see also \cite{BCD} and \cite{G} for earlier results). We show that, for a large family of Young functions, the condition presented in \cite{Z} is necessary and sufficient for weak amenability of weighted Orlicz algebras.
As an example, we apply our results to $\Z^d$, the group of $d$-dimensional integers, and characterize when $l^p_\om(\Z^d)$ ($1<p<\infty$) is weakly amenable. We also present many more classes of (non-) weakly amenable weighted Orlicz algebras.

We finish the introductory part by pointing out that throughout this paper we concern ourselves with the
theory of ``bounded multiplications" for Banach algebras and Banach modules, as opposed to ``contractive multiplications". Also weights for us are ``weakly submultiplicative" as opposed to ``submultiplicative".

\section{Preliminaries}

In this section, we give some definitions and state some technical results that will be crucial in the rest of the paper. In what follows, $G$ always denotes a locally compact group with a fixed left Haar measure $ds$.

\subsection{Orlicz Spaces}

We first recall some facts concerning Young functions and Orlicz spaces. Our main reference is \cite{rao}.

A nonzero function $\Phi:[0,\infty) \to[0,\infty]$ is called a Young
function if $\Phi$ is convex, $\Phi(0)=0$, and $\lim_{x\to \infty} \Phi(x)=\infty$. For a Young function $\Phi$, the complementary function $\Psi$ of
$\Phi$ is given by
\begin{align}\label{Eq:Young function-complementary}
\Psi(y)=\sup\{xy-\Phi(x):x\ge0\}\quad(y\geq 0).
\end{align}
It is easy to check that $\Psi$ is also a Young function. Also, if $\Psi$ is the complementary function of $\Phi$, then $\Phi$ is
the complementary function of $\Psi$ and $(\Phi,\Psi)$ is called a
complementary pair. We always have the Young inequality
\begin{align}\label{Eq:Young inequality}
xy\le\Phi(x)+\Psi(y)\quad(x,y\ge0)
\end{align}
for complementary functions $\Phi$ and $\Psi$. By
definition, Young function can have the value $\infty$ at some point,
and hence be discontinuous at this point. However, we always consider the pair of complementary Young
functions $(\Phi,\Psi)$ with both $\Phi$ and $\Psi$ being continuous and strictly increasing. In particular, they attain positive values on $(0,\infty)$. In this case, if we let $\Phi^{-1}$ and $\Psi^{-1}$ to be the inverse functions of $\Phi$ and $\Psi$, respectively, then, by \cite[Lemma~4.8.16]{BS} we have
\begin{align}\label{Eq:Inverse Young inequality}
x\leq \Phi^{-1}(x)\Psi^{-1}(x)\leq 2x \ \ \ (0\leq x<\infty).
\end{align}
Now suppose that $G$ is a locally compact group with a fixed Haar measure $ds$ and $(\Phi,\Psi)$ is a complementary pair of Young functions. We define
\begin{align}\label{Eq:Orlicz defn-0}
\mathcal{L}^\Phi(G)=\left\{f:G\to\C:f \ \text{is measurable and}\  \int_G\Phi(|f(s)|)\,ds <\infty
\right\}.
\end{align}
Since $\mathcal{L}^\Phi(G)$ is not always a linear space, we define the Orlicz space $L^\Phi(G)$ to be
\begin{align}\label{Eq:Orlicz defn}
L^\Phi(G)=\left\{f:G\to\C:\int_G\Phi(\alpha|f(s)|)\,ds <\infty
\mbox{ for some }\alpha>0\right\},
\end{align}
where $f$ indicates a member in equivalence class of measurable functions with respect to the Haar measure $ds$. When $G$ is discrete, we simply use the standard terminology and write $l^\Phi(G)$ instead of $L^\Phi(G)$. The Orlicz space is a Banach space under the (Orlicz) norm $\|\cdot\|_\Phi$
defined for $f\in L^\Phi(G)$ by
\begin{align}\label{Eq:Orlicz norm}
\|f\|_\Phi=\sup\left\{\int_G|f(s)v(s)|\,ds: \int_G\Psi(|v(s)|)\,ds \le1\right\},
\end{align}
where $\Psi$ is the complementary function of $\Phi$. One can also
define the (Luxemburg) norm $N_\Phi(\cdot)$ on $L^\Phi(G)$ by
\begin{align}\label{Eq:Orlicz Luxemburg defn}
N_\Phi(f)=\inf\left\{k>0:\int_G\Phi\left(\frac{|f(s)|}{k}\right)
\,ds \le1\right\}.
\end{align}
It is known that these two norms are equivalent, i.e.,
\begin{align}\label{Eq:Orlicz norm-Luxemburg relation}
N_\Phi(\cdot)\le \|\cdot\|_\Phi\le2 N_\Phi(\cdot)
\end{align}
and
\begin{align}\label{Eq:Orlicz norm-defn relation}
N_\Phi(f)\le1 \ \ \text{if and only if}\ \ \int_G\Phi(|f(s)|)\,ds \le1.
\end{align}
Let $\Sm^\Phi(G)$ be the closure of the linear
space of all step functions in $L^\Phi(G)$. Then $\Sm^\Phi(G)$ is a Banach space
and contains $C_c(G)$, the space of all continuous functions on $G$ with compact support, as a dense subspace \cite[Proposition 3.4.3]{rao}. Moreover, $\Sm^\Phi(G)^*$, the dual of  $\Sm^\Phi(G)$, can be identified with $L^\Psi(G)$ in a natural way \cite[Theorem 4.1.6]{rao}. Another useful characterization of $\Sm^\Phi(G)$ is that
$f\in \Sm^\Phi(G)$ if and only if for every $\alpha>0$, $\alpha f\in \mathcal{L}^\Phi(G)$ \cite[Definition 3.4.2 and Proposition 3.4.3]{rao}.

A Young function $\Phi$ satisfies the $\Delta_2$-condition 
if there
exists a constant $K>0$
such that
\begin{align}\label{Eq:Delta 2 condition}
\Phi(2x)\le K\Phi(x) \ \ \text{for all}\ \ x\ge 0.
\end{align}
In this case we write
$\Phi\in\Delta_2$.
If $\Phi\in\Delta_2$, then it follows that $L^\Phi(G)=\Sm^\Phi(G)$ so that  $L^\Phi(G)^*=L^\Psi(G)$ \cite[Corollary 3.4.5]{rao}. If, in addition, $\Psi\in\Delta_2$, then the Orlicz space $L^\Phi(G)$ is a reflexive Banach space.

As in \cite[Page 20]{rao}, we say that two Young functions $\Phi_1$ and $\Phi_2$ are {\it strongly equivalent} and write $\Phi_1 \approx \Phi_2$ if there exists $0<a\leq b<\infty$ such that
\begin{equation}\label{Eq:equivalent young functions}
\Phi_1(ax)\leq \Phi_2(x)\leq \Phi_1(bx) \ \ \ (x\geq 0).
\end{equation}
It is clear from the definition of the Orlicz space \eqref{Eq:Orlicz defn} that the strongly equivalent Young functions generate the same Orlicz space.

We will frequently use the (generalized) H\"{o}lder's inequality for Orlicz spaces
\cite[Remark 3.3.1]{rao}.
More precisely, for any complementary pair of Young functions $(\Phi,\Psi)$
and any $f\in L^\Phi(G)$ and $g\in L^\Psi(G)$, we have
\begin{align}\label{Eq:Holder inequality}
\|fg\|_1:=\int_G |f(s)g(s)|ds \leq \min\{N_\Phi(f)\|g\|_\Psi , \|f\|_\Phi N_\Psi(g)\}.
\end{align}
This, in particular, implies that $fg\in L^1(G)$.

In general, there is a straightforward method to construct various complementary pairs of strictly increasing continuous Young functions as described in
\cite[Theorem 1.3.3]{rao}. Suppose that $\varphi: [0,\infty)\to [0,\infty)$ is a continuous strictly increasing function with $\varphi(0)=0$ and
$\lim_{x\to \infty} \varphi(x)=\infty.$
Then $$\Phi(x)=\int_0^x \varphi(y)dy$$ is a continuous strictly increasing Young function and
$$\Psi(y)=\int_0^y \varphi^{-1}(x)dx$$
is the complementary Young function of $\Phi$ which is also continuous and strictly increasing.
Here $\varphi^{-1}(x)$ is the inverse function of $\varphi$. Below are several families of examples obtained using the above construction (see \cite[Proposition 2.11]{ML} and \cite[Page 15]{rao} for more details):

$(1)$ For $1< p,q<\infty$ with $\frac{1}{p}+\frac{1}{q}=1$, if $\Phi(x)=\frac{x^p}{p}$, then $\Psi(y)=\frac{y^q}{q}$. In this case,
the space $L^{\Phi}(G)$ becomes the Lebesgue space $L^p(G)$ and the norm $\|\cdot\|_{\Phi}$ is equivalent to the
classical norm $\|\cdot\|_{p}$.


(2) If $\Phi(x)=\cosh x-1$, then $\Psi(x)\approx x\ln (1+x)$.

(3) If $\Phi(x)=(1+x)\ln(1+x)-x$, then $\Psi(x)=e^x-x-1$.



\subsection{Weighted Orlicz algebras}\label{S:Twisted Orlicz alg}

In this section, we present and summarize what we need from the theory of twisted Orlicz algebras. These are taken from \cite{OS1}.


A {\bf weight} on $G$ is a locally integrable measurable function $\om : G \to \R_+$
with $\om(e)=1$ and $1/\om\in L^\infty(G)$ such that there is $C>0$ satisfying
\begin{align}\label{Eq:defn-weight}
\om(st)\leq C\om(s)\om(t) \ \ \ (s,t\in G)
\end{align}
In this case, we define
$$\Om(s,t)=\frac{\om(st)}{\om(s)\om(t)} \ \ \ (s,t\in G).$$
For a weight $\om$ on $G$, we define the weighted spaces
$$L^\Phi_\om(G):=\{f: f\om \in L^\Phi(G) \},$$
and
$$\Sm^\Phi_\om(G):=\{f: f\om \in \Sm^\Phi(G) \}.$$
It is clear that with the norm
$$\|f\|_{\Phi,\om}:=\|f\om\|_\Phi,$$
$L^\Phi_\om(G)$ becomes a Banach space having $\Sm^\Phi_\om(G)$ as a closed subspace. We say that $L^\Phi_\om(G)$ is a {\bf weighted Orlicz algebra} if it is a Banach algebra under the convolution product.  In this case, $\Sm^\Phi_\om(G)$ becomes a closed subalgebra of $L^\Phi_\om(G)$ which we call the {\bf maximal essential subalgebra of} $L^\Phi_\om(G)$.

In \cite{OS1}, sufficient conditions on $\om$ were found under which the convolution turns a weighted Orlicz space into a Banach algebra (see \cite[Lemma 3.2 and Theorem 3.3]{OS1}). We summarize them below:

\begin{thm}\label{T:twisted Orlicz alg}
Let $G$ be a locally compact group and $\om$ a weight on $G$. \\
$(i)$ $L_\om^\Phi(G)$ is a Banach $L_\om^1(G)$-bimodule with respect to the convolution having $\Sm_\om^\Phi(G)$ as an essential Banach $L_\om^1(G)$-submodule.\\
$(ii)$ Suppose that there exist non-negative measurable functions $u$ and $v$
in $L^\Psi(G)$ such that
\begin{align}\label{Eq:2-cocycle bdd sum}
|\Om(s,t)|\leq u(s)+v(t) \ \ \ (s,t\in G).
\end{align}
Then for every $f,g\in L_\om^\Phi(G)$, the convolution is well-defined on $L_\om^\Phi(G)$
so that $(L_\om^\Phi(G),\tw)$ becomes a twisted Orlicz algebra having $\Sm_\om^\Phi(G)$ as a closed subalgebra.
\end{thm}

Let $G$ be a compactly generated abelian group. Then, by the Structure Theorem,
\begin{align}\label{Eq:Compactly gen abelian group-structure thm}
G\cong \R^k \times \Z^m \times T
 \end{align}
 where $k,m \in \N \cup \{0\}$ and $T$ is a compact (abelian) group. We can then consider the
 generating set
 \begin{align}\label{Eq:Compactly gen abelian group-generating set}
 U=(-1,1)^k\times \{-1,0,1\}^m\times T
 \end{align}
 and define a {\it length function} $|\cdot|_U : G \to \N\cup \{0\}$ by
\begin{align}\label{Eq:length function}
|s|_U=\inf \{n\in \N : s\in U^n \} \ \ \text{for} \ \ s \neq e, \ \ |e|=0.
\end{align}
When there is no fear of ambiguity, we write $|\cdot|$ instead of $|\cdot|_U$.
It is straightforward to verify that $|\cdot|$ is a symmetric subadditive function on $G$, i.e.
\begin{align}\label{Eq:lenght func-trai equality}
|s+t|\leq |s|+|t| \ \ \text{and} \ \ |s|=|-s|\ \ \ (s,t\in G).
\end{align}
Now if $\nu:\N\cup \{0\}\to \R^+$ is a continuous increasing subadditive function with $\nu(0)=0$ and $\lim_{n\to \infty}\nu(n)=\infty$, then
\begin{align}\label{Eq:weight-lenght function}
\om(t)=e^{\nu(|t|)} \ \ \ (t\in G)
\end{align}
is a weight on $G$. This gives rise to a number of different weights on $G$.
For example, for every $0< \alpha \leq 1$, $\beta \geq 0$, $\gamma >0$, and $C>0$,
we can define the {\it polynomial weight} $\om_\beta$ on $G$ of order $\beta$ by
\begin{align}\label{Eq:poly weight-defn}
\om_\beta(s)=(1+|s|)^\beta  \ \ \ \ (s\in G),
\end{align}
and the {\it subexponential weights} $\sg_{\alpha, C}$ and $\rho_{\beta,C}$ on $G$ by
\begin{align}\label{Eq:Expo weight-defn}
\sg_{\alpha,C}(s)=e^{C|s|^\alpha} \ \ \ \ (s\in G)
\end{align}
and
\begin{align}\label{Eq:Expo weight II-defn}
\rho_{\gamma,C}(s)=e^\frac{C|s|}{(\ln (1+|s|))^\gamma} \ \ \ \ (s\in G).
\end{align}

We can apply Theorem~\ref{T:twisted Orlicz alg}(ii) to these weights to obtain weighted Orlicz algebras. In particular, we have the following result (\cite[Corollary 5.3 and Theorem 5.8]{OS1}).

\begin{thm}\label{T:twisted Orlicz alg-poly and exp weight-Poly growth}
Let $G$ be a compactly generated abelian group, and $\om$ be a weight on $G$.
Then $(L_\om^\Phi(G),\tw)$ is a weighted Orlicz algebra if $\om$ is either one of the following weights:\\
$(i)$ $\om=\om_\beta$, the polynomial weight \eqref{Eq:poly weight-defn} with $1/\om \in \Sm^\Psi(G)$;\\
$(ii)$ $\om=\sg_{\alpha,C}$, the subexponential weight \eqref{Eq:Expo weight-defn};\\
$(iii)$ $\om=\rho_{\gamma,C}$, the subexponential weight \eqref{Eq:Expo weight II-defn}.
\end{thm}

We finish this section by pointing out that in the proof of \cite[Corollary 5.3]{OS1}, it was shown that $1/\om_\beta \in \Sm^\Psi(G)$ if $\beta> \frac{k+m}{l}$,
where $k$ and $m$ come from the representation \eqref{Eq:Compactly gen abelian group-structure thm} and
$l\geq 1$ is such that $\lim_{x\to 0^+}\frac{\Psi(x)}{x^l}$ exists.






\section{Weak amenability}

A Banach algebra $A$ is {\it weakly amenable} if every bounded derivation $D$ from $A$ into $A^*$ is inner \cite[Definition 4.2.1]{Run}. When $A$ is commutative, one can use the equivalent formulation that $A$ is weakly amenable if every bounded derivation from $A$ into any symmetric Banach $A$-module is zero \cite[Theorem 2.8.63]{D}. The class of weakly amenable Banach algebras is much larger than the class of amenable ones. For instance, all C$^*$-algebras and group algebras are weakly amenable \cite[Theorems 4.2.3 and 4.2.4]{Run}.

In \cite[Thereom 3.1]{Z}, Y. Zhang has found a necessary and sufficient condition, formulated below in \eqref{Eq:group homo vs weak amen}, for weak amenability of weighted group algebras on abelian groups. His work extends the previous partial results in this direction (see \cite{Z} and the reference therein). Our goal is to further extend the result of Zhang to weighted Orlicz algebras. Interestingly, we will see that, in most cases, the criterion found in \cite[Theorem 3.1]{Z} also works in our settings.

We start with the following theorem which shows that the sufficient condition formulated in \cite{Z} will imply weak amenability of the maximal essential subalgebras of weighted Orlicz algebras. Note that in what follows we always consider $\mathbb{C}$ as an additive group.

\begin{thm}\label{T:weighted Orlicz algebra-weak amenable}
Let $G$ be a locally compact abelian group, $\om$ be a weight on $G$, and $(\Phi,\Psi)$ be a complementary pair of Young functions such that $(L^\Phi_\om(G),*)$ is a Banach algebra. Suppose that there exists no nonzero continuous group
homomorphism $\xi:G\to \C$ such that $\txi \in L^\infty(G)$, where
\begin{align}\label{Eq:unbounded group homo vs weak amen}
\txi(s):=\frac{\xi(s)}{\om(s)\om(s^{-1})} \ \ \ \ (s\in G).
\end{align}
Then  $\Sm^\Phi_\om(G)$ is weakly amenable.
\end{thm}

\begin{proof}
We first note that, by the main result of \cite[Theorem 3.1]{Z}, the nonexistence of group homomorphisms $\xi:G\to \C$ satisfying \eqref{Eq:unbounded group homo vs weak amen} will imply (in fact, it is equivalent to) the weak amenability of $L^1_\om(G)$. Now consider the space
$$\Sm_\om^{1,\Phi}(G):=L_\om^1(G)\cap \Sm_\om^\Phi(G).$$
It is clear that $\Sm_\om^{1,\Phi}(G)$ with the norm
$$\|\cdot\|_\om^{1,\Phi}:=\|\cdot\|_{1,\om}+\|\cdot\|_{\Phi,\om}$$
is a Banach space. Moreover, it follows from \cite[Lemma 3.2]{OS1} and \cite[Lemma 3.1]{OS2} that
$\Sm_\om^{1,\Phi}(G)$ becomes an abstract Segal subalgebra of $L_\om^1(G)$ in the sense of \cite[Definition 4.1.8]{D}. Hence, by \cite[Theorem 4.1.10]{D}, $\Sm_\om^{1,\Phi}(G)$ is weakly amenable. In particular, this implies that $\Sm_\om^{\Phi}(G)$ is weakly amenable by \cite[Proposition 2.8.64]{D} and the fact that $\Sm_\om^{1,\Phi}(G)$ is dense in $\Sm_\om^\Phi(G)$.



\end{proof}


The rest of this section is devoted to showing that, in many instances, the condition in Theorem \ref{T:weighted Orlicz algebra-weak amenable} is also necessary for weak amenability of weighted Orlicz algebras and their maximal essential subalgebras. After proving a technical lemma, the main results will be presented in Theorems~\ref{T:weighted Orlicz algebra-non weak amenable}~and~\ref{T:weighted Orlicz algebra-non weak amenable-compactly generated group}.


\begin{lem}\label{L:Young functions-square root}
Let $(\Phi,\Psi)$ be a complementary pair of Young functions.
Suppose that $\tPs(x):=\Psi(\sqrt{x})$ is a Young function and $\tPh$ is the complementary Young function of $\tPs$. Then\\
$(i)$ for every $x>0$,
\begin{eqnarray}\label{Eq:Young functions-square root-comparision I}
\Phi(x)\leq \tPh\left(\frac{2x^2}{\Phi(x)}\right), \ \  \ \tPh\left(\frac{x^2}{4\Phi(x)}\right)\leq \Phi(x).
\end{eqnarray}
$(ii)$ We have the following jointly continuous inclusions:
\begin{eqnarray}\label{Eq:Orlicz spaces-convolution product}
L^{\tPs}(G)*L^{\Phi}(G) \subseteq L^{\Psi}(G),
\end{eqnarray}
\begin{eqnarray}\label{Eq:Orlicz spaces-pointwise product}
L^{\tPh}(G)L^{\Psi}(G) \subseteq L^{\Phi}(G).
\end{eqnarray}
\end{lem}

\begin{proof}
Since $\Phi$ and $\Psi$ are strictly increasing, then so are their inverse functions. Hence, it is clear
that inequalities \eqref{Eq:Young functions-square root-comparision I} are equivalent to
\begin{eqnarray}\label{Eq:Young functions-square root-comparision II}
\frac{\tPh^{-1}(x)}{2}\leq \frac{\Phi^{-1}(x)^2}{x},\  \ \ \frac{\Phi^{-1}(x)^2}{x}\leq 4\tPh^{-1}(x)\quad(x>0).
\end{eqnarray}
Since $\tPs(x):=\Psi(\sqrt{x})$, we have that
\begin{eqnarray}\label{Eq:Young functions-square root-comparision III}
\tPs^{-1}(x)=\Psi^{-1}(x)^2  \ \ \ (x\geq 0).
\end{eqnarray}
Combining \eqref{Eq:Inverse Young inequality} and \eqref{Eq:Young functions-square root-comparision III}, we obtain \eqref{Eq:Young functions-square root-comparision II}
and the following inequality:
\begin{equation}\label{Eq:5}
\tPs^{-1}(x) \Phi^{-1}(x) \leq 2 x\Psi^{-1}(x) \ \ \ (x\geq 0).
\end{equation}
According to \cite[Theorem 3.3.9]{rao}, if Young functions $\Phi_1, \,\Phi_2,\,\Phi_3$ are such that
$$\Phi_1^{-1}(x) \Phi_2^{-1}(x) \leq x\Phi_3^{-1}(x) \ \ \ (x\geq 0), $$
then $L^{\Phi_1}*L^{\Phi_2}\subseteq L^{\Phi_3}$, and the embedding is jointly continuous. Because of (\ref{Eq:5}), the above condition is satisfied for $\Phi_1=\tilde{\Psi}$, $\Phi_2=\Phi$, and $\Phi_3(x)=\Psi(x/2)$. So in order to prove (\ref{Eq:Orlicz spaces-convolution product}), we just note that $\Phi_3$ is strongly equivalent to $\Psi$ in the sense of (\ref{Eq:equivalent young functions}), and hence they produce the same Orlicz space.

Finally, using \eqref{Eq:Inverse Young inequality} and (\ref{Eq:Young functions-square root-comparision III}), we can also get

$$\tPh^{-1}(x) \Psi^{-1}(x) \leq \frac{2x\Psi^{-1}(x)}{\tPs^{-1}(x)}=\frac{2x}{\Psi^{-1}(x)}\leq 2\Phi^{-1}(x)\ \ \ (x\ge0),$$
which implies the jointly continuous inclusion
$$L^{\tPh}(G)L^\Psi(G)\subseteq L^\Phi(G)$$
by \cite[Theorem 3.3.7]{rao} and the same kind of argument as above to get rid of the constant $2$.

\end{proof}




\begin{thm}\label{T:weighted Orlicz algebra-non weak amenable}
Let $G$ be a locally compact abelian group, $\om$ be a weight on~$G$, and $(\Phi,\Psi)$ be a complementary pair of Young functions. Suppose that there exists a nonzero continuous group homomorphism $\xi:G\to \C$ such that $\txi \in L^\infty(G)$, where $\txi$ is defined in \eqref{Eq:unbounded group homo vs weak amen}. If $(L^\Phi_\om(G),*)$ is a Banach algebra, then neither $L^\Phi_\om(G)$ nor  $\Sm^\Phi_\om(G)$  is weakly amenable in any of the following cases:\\
$(i)$ $\tPs(x):=\Psi(\sqrt{x})$ is a Young function;\\
$(ii)$ $\tPh(x):=\Phi(\sqrt{x})$ is a Young function and $\txi\in L^{\tPs}$, where $\tPs$ is the complementary Young function of $\tPh$.
\end{thm}

\begin{proof}
   Let $\xi:G\to \C$ be a nonzero continuous group homomorphism such that $\txi \in L^\infty(G)$, where $\txi$ is defined in \eqref{Eq:unbounded group homo vs weak amen}.
Let $U$ be a compact neighborhood of the identity of $G$. For every $f\in C_c(G)$, define
$$D(f)=1_U*(\check{f}\check{\xi}),$$
where, as usually, $\check{g}(x)=g(x^{-1})$ $(x\in G)$. It is easy to see that $D$ is a linear operator that ranges in $L^\Phi_{\om}(G)^*$. Moreover, a similar argument to the one presented in \cite[Theorem 3.1]{Z} shows that $D$ is a non-zero derivation on $C_c(G)$. Hence it suffices to show that $D$ can be extended to a bounded linear operator on  $L_\om^\Phi(G)$.

For every $g\in L_\om^\Phi(G)$, we have
\begin{eqnarray*}
\la D(f) \,, \,g \ra &=& \int_G\int_G 1_U(ts)f(s)\xi(s)g(t) dsdt \\
&=&  \int_G\int_G \frac{1_U(ts)\om(s^{-1})}{\om(t)}(f\om)(s)\txi(s)(g\om)(t) dsdt.
\end{eqnarray*}
Applying the inequality $\om(s^{-1})\leq C \om(s^{-1}t^{-1})\om(t)$ (see \eqref{Eq:defn-weight}) and the Fubini's theorem, we get
\begin{eqnarray*}
|\la D(f) \,, \,g \ra| &\leq &  C\int_G\int_G 1_U(ts)\om((ts)^{-1})|f\om|(s)|\txi(s)||g\om|(t) dtds \\
&=& C\,\la 1_U\check{\om}*|\check{g}\check{\om}| \ , \, |f\om\txi| \ra.
\end{eqnarray*}
Now, if (i) holds, then by \eqref{Eq:Orlicz spaces-convolution product} of Lemma~\ref{L:Young functions-square root} we have

$$L^{\tPs}(G)* L^\Phi(G)\subseteq L^\Psi(G).$$
Since $f\om,\, \check{g}\check{\om} \in L^\Phi(G)$, $\txi \in L^\infty(G)$, and $1_U\check{\om} \in L^{\tPs}(G)$, it follows that  $1_U\check{\om}*|\check{g}\check{\om}|\in L^\Psi(G)$ and $f\om\txi \in L^\Phi(G)$, so that
$$ |\la D(f) \,, \,g \ra| \leq C_1 \|1_U\|_{\tPs,\check{\om}} \|\txi\|_\infty \|g\|_{\Phi,\om}\|f\|_{\Phi,\om}.$$
Now suppose that (ii) holds. By \eqref{Eq:Orlicz spaces-pointwise product} of Lemma~\ref{L:Young functions-square root} we have that
$$L^{\tPs}(G)L^\Phi(G)\subseteq L^\Psi(G).$$
Since $f\om,\, \check{g}\check{\om} \in L^\Phi(G)$, $1_U\check{\om}\in L^1(G)$, and $\txi  \in L^{\tPs}(G)$, it follows that  $1_U\check{\om}*|\check{g}\check{\om}|\in L^\Phi(G)$ and $f\om\txi \in L^\Psi(G)$, so that
$$ |\la D(f) \,, \,g \ra| \leq C_2 \|1_U\|_{1,\check{\om}} \|\txi\|_{\tPs} \|g\|_{\Phi,\om}\|f\|_{\Phi,\om}.$$
Thus, in either case, $D$ has a continuous extension to a linear operator from $L^\Phi_\om(G)$ to $L^\Phi_{\om}(G)^*$, and so neither $L^\Phi_\om(G)$ nor $\Sm_\om^\Phi(G)$ is weakly amenable.
\end{proof}

\begin{thm}\label{T:weighted Orlicz algebra-non weak amenable-compactly generated group}
Let $G$ be a non-compact compactly generated abelian group, $(\Phi,\Psi)$ be a complementary pair of Young functions, and $\om$ be the weight on $G$ defined in \eqref{Eq:weight-lenght function}. Suppose that $(L^\Phi_\om(G),*)$ is a Banach algebra and $1/\om \in L^\Psi(G)$. If $\tPh(x):=\Phi(\sqrt{x})$ is a Young function, then neither $L^\Phi_\om(G)$ nor  $\Sm^\Phi_\om(G)$  is weakly amenable.
\end{thm}

\begin{proof}
Let $\xi:G\to \C$ be any nonzero continuous group homomorphism (its existence follows, for example, from the structural
description of $G$) and let $\txi$ be the function defined in \eqref{Eq:unbounded group homo vs weak amen}. By Theorem \ref{T:weighted Orlicz algebra-non weak amenable}, it suffices to prove that $\txi\in L^{\tPs}(G)$.
Our assumption that $1/\om \in L^\Psi(G)$ means that there exists $\alpha>0$ such that
\begin{align}\label{Eq:1}
\int_G \Psi\left(\frac{\alpha}{\om(s)}\right)ds < \infty.
\end{align}
Now put
\begin{align}\label{Eq:2}
f(s)=\frac{\alpha}{\om(s)}:=\frac{\alpha}{e^{\nu(|s|)}} \ \ \ (s\in G).
\end{align}
We first prove that
\begin{align}\label{Eq:3}
\xi (\Psi\circ f) \in L^\infty(G).
\end{align}
Since $\xi$ is a group homomorphism, for every $n\in \N$ and $s\in U^n \setminus U^{n-1}$, we have
\begin{align}\label{Eq:4}
\xi(s) \Psi(f(s)) \leq Cn a_n,
\end{align}
where
$$C=\sup \{\xi(s): s\in U \} \ \ \text{and} \ \ a_n=\Psi\Biggl(\frac{\alpha}{e^{\nu(n)}}\Biggr).$$
On the other hand, $\nu$ and $\Psi$ are both increasing so that $\{a_n\}_{n\in \N}$ is a decreasing sequence of positive numbers. Therefore, if the Haar measure of $G$ is denoted by $\lambda$ and we assume that $\lambda(U)=1$, then
\begin{eqnarray*}
C^{-1}\xi(s) \Psi(f(s)) &\leq& n a_n \ \ \ (\text{by}\ \eqref{Eq:4}) \\
&\leq& a_1+\cdots+ a_n \\
&\leq &  \sum_{n=1}^\infty \displaystyle \lambda(U^n \setminus U^{n-1})a_n \\ 
&=& \sum_{n=1}^\infty \int_{U^n\setminus U^{n-1}} \Psi(f(t)) dt \\
&\leq& \int_{G} \Psi(f(t)) dt <  \infty \ \ \ (\text{by}\ \eqref{Eq:1}\ \text{and}\ \eqref{Eq:2}).
\end{eqnarray*}
Hence, \eqref{Eq:3} is verified. Next, since $\tPh(x):=\Phi(\sqrt{x})$ is a Young function with the complementary function $\tPs$, it follows from \eqref{Eq:Young functions-square root-comparision I} of Lemma \ref{L:Young functions-square root}, \eqref{Eq:1}, and \eqref{Eq:2} that
$$\int_G \tPs\left(\frac{f(t)^2}{4\Psi(f(t))}\right) dt \leq \int_G \Psi(f(t)) dt<\infty.$$
Hence, by definition \eqref{Eq:Orlicz defn}, $\frac{f^2}{\Psi(f)}\in L^{\tPs}(G)$. This, together with \eqref{Eq:3}, implies that
$$\xi f^2=\xi \Psi(f)\frac{f^2}{\Psi(f)}\in L^\infty(G)L^{\tPs}(G)\subseteq L^{\tPs}(G).$$
Finally, $\txi\in L^{\tPs}(G)$ since $$\txi=\frac1{\alpha^2} \xi f^2.$$
\end{proof}

We can summarize the preceding results and also provide an equivalent and easy-to-check criterion on Young functions so that Theorems \ref{T:weighted Orlicz algebra-non weak amenable} and \ref{T:weighted Orlicz algebra-non weak amenable-compactly generated group} may be applied.

\begin{cor}\label{C:weighted lp algebra-(non) weak amenable1}
Let $G$ be a non-compact, compactly generated abelian group, $\om$ be a weight on $G$, and $(\Phi,\Psi)$ be a complementary pair of Young functions such that $\Psi'$ exists on $\R^+$ and $\Psi'(x)/x$ is increasing on $\R^+$.\\
$(i)$ If $(L^\Phi_\om(G),*)$ is a Banach algebra, then  $\Sm^\Phi_\om(G)$  is weakly amenable if and only if there exists no nonzero continuous group
homomorphism $\xi:G\to \C$ such that $\txi \in L^\infty(G)$, where
\begin{equation*}\label{Eq:group homo vs weak amen}
\txi(s):=\frac{\xi(s)}{\om(s)\om(s^{-1})} \ \ \ \ (s\in G).
\end{equation*}
$(ii)$ If $(L^\Psi_\om(G),*)$ is a Banach algebra, $1/\om \in L^\Phi(G)$, and $\om$ is of the form \eqref{Eq:weight-lenght function}, then neither $L^\Psi_\om(G)$ nor  $\Sm^\Psi_\om(G)$ is weakly amenable.
\end{cor}
\begin{proof}
Note that if $\tilde{\Psi}(x)=\Psi(\sqrt x)$, then $\tilde{\Psi}'(x)=\frac{\Psi'(\sqrt x)}{2\sqrt x}$ is increasing because we assumed that ${\Psi'(x)}/{x}$ is increasing. Therefore, $\tilde{\Psi}$ is convex and hence it is a Young function. We then immediately obtain (ii) from Theorem~\ref{T:weighted Orlicz algebra-non weak amenable-compactly generated group} and one direction of (i)~--- from Theorem~\ref{T:weighted Orlicz algebra-non weak amenable}(i). The other direction of (i) follows from Theorem~\ref{T:weighted Orlicz algebra-weak amenable}.
\end{proof}

\begin{rem}
It is straightforward to verify that Corollary~\ref{C:weighted lp algebra-(non) weak amenable1} can be applied to various Young functions to determine the weak amenability of their weighted Orlicz algebras. Below, we give few families of examples of such Young functions ($p\geq 1$ is arbitrary):

(1) $\Psi(x)=[x^2\ln (1+x)]^p$.

(2) $\Psi(x)=(e^x-x-1)^p$.

(3) $\Psi(x)=(e^{x^p}-1)$

(4) $\Psi(x)=(\cosh x-1)^p$.

\end{rem}

We finish this section by highlighting, in the following corollary and the example afterward, that when weighted Orlicz spaces are just weighted $L^p$ spaces, we have fairly general results on their weak amenability.

\begin{cor}\label{C:weighted lp algebra-(non) weak amenable}
Let $G$ be a non-compact compactly generated abelian group, $\om$ be a weight on $G$, and $1<p,\,q<\infty$ be such that $1/p+1/q=1$. Suppose that $(L^p_\om(G),*)$ is a Banach algebra. \\
$(i)$ For $1<p\le 2$, $L^p_\om(G)$ is weakly amenable if and only if there exists no nonzero continuous group
homomorphism $\xi:G\to \C$ such that $\txi \in L^\infty(G)$, where
\begin{equation*}\label{Eq:group homo vs weak amen}
\txi(s):=\frac{\xi(s)}{\om(s)\om(s^{-1})} \ \ \ \ (s\in G).
\end{equation*}
$(ii)$ For $p>2$, if $\om$ is of the form \eqref{Eq:weight-lenght function} and $1/\omega\in L^q(G)$, then
$L^p_\om(G)$ is not weakly amenable.
\end{cor}

\begin{proof} As was mentioned in the Preliminaries, $L^p_\om(G)=L^{\Phi}_{\omega}(G)$ for $\Phi(x)={x^p}/{p}$ and the corresponding complementary Young function is $\Psi(x)=x^q/q$. Since $\Phi\in\Delta_2$, we have that $L^p_\om(G)=L^{\Phi}_{\omega}(G)=\Sm^{\Phi}_{\omega}(G)$. For $1<p\le2$ we have $q\ge2$ implying that $\Psi'(x)/x=x^{q-2}$ is increasing, and so (i) follows directly from Corollary~\ref{C:weighted lp algebra-(non) weak amenable1}(i).

Likewise, to obtain (ii), we apply Corollary~\ref{C:weighted lp algebra-(non) weak amenable1}(ii) to $\Phi(x)={x^q}/{q}$ and $\Psi(x)=x^p/p$.
\end{proof}

\begin{exm}\label{E:twisted lp alg on integers-poly and exp weight-Operator alg}
Let $\Z^d$ be the group of $d$-dimensional integers.
The standard choice of generating set for $\Z^d$ is
	$$F=\{(x_1,\ldots, x_d) \mid x_i\in \{-1,0,1\} \},$$
It is straightforward to verify that
\begin{align}\label{Eq:growth generating set of integers}
|F^n|=(2n+1)^d \ \ \ (n=0,1,2,\ldots).
\end{align}
Now suppose $1< p,q<\infty$ with $\frac{1}{p}+\frac{1}{q}=1$. Let $\om_\beta$ be the polynomial weight on $\Z^d$ defined in \eqref{Eq:poly weight-defn}. Then
\begin{align}\label{Eq:poly weight-d dim integers}
\om_\beta(\mathbf{x})=(1+\max\{|x_i|:i=1,\ldots,d\})^\beta \ \ \ (\mathbf{x}=(x_1,\ldots,x_d)\in \Z^d).
\end{align}
Hence, by Theorem \ref{T:twisted Orlicz alg-poly and exp weight-Poly growth} and \cite[Theorem 3.2]{K3}, $l^p_{\om_\beta}(\Z^d)$ is a Banach algebra if and only if $1/\om_\beta \in l^q(\Z^d)$ and, by \eqref{Eq:growth generating set of integers}, the latter happens exactly when $\beta >d/q$.
Therefore, by Corollary \ref{C:weighted lp algebra-(non) weak amenable}, $l^p_{\om_\beta}(\Z^d)$ is weakly amenable if $1<p<2$ and
$d/q <\beta <1/2$
 and $l^p_{\om_\beta}(\Z^d)$ is not weakly amenable if $p> 2$ and $\beta >d/q$.
On the other hand, if we let $\om$ to be either of the subexponential weights
defined in \eqref{Eq:Expo weight-defn} or \eqref{Eq:Expo weight II-defn}, then $\txi\in l^\infty(\Z^d)$ (defined in \eqref{Eq:group homo vs weak amen}) for any nonzero group homomorphism $\xi: \Z^d \to \C$ . Thus, again by applying Corollary \ref{C:weighted lp algebra-(non) weak amenable}, it follows that $l^p_\om(\Z^d)$ is not weakly amenable. One can also obtain similar statements for $L^p_\om(\R^d)$.
\end{exm}


\section{Acknowledgements}
This work was initiated while the second named author visited Istanbul University under a fellowship from Tubitak for visiting scientists and scientists on sabbatical leave. The first named author would like to acknowledge the support of Tubitak and express his deep gratitude toward his host, Serap
\"{O}ztop. Also, the third named author would like to acknowledge the support from PIMS Institute as she is currently holding a PIMS postdoctoral Fellowship. The authors would like to thank the referee for carefully reading the paper and providing comments that have improved the exposition of the paper. This includes, in particular, finding a minor error in the earlier version of Corollary \ref{C:weighted lp algebra-(non) weak amenable}.

\end{document}